\documentclass[11pt,reqno]{amsart}

\usepackage[T1]{fontenc}
\usepackage{lmodern}
\usepackage{microtype}
\usepackage{amsmath,amssymb,mathtools}
\usepackage{amsthm}
\usepackage{aliascnt}
\usepackage{booktabs}
\usepackage{array}
\usepackage{enumitem}
\usepackage{tikz}
\usepackage{float}
\usepackage{hyperref}
\usepackage[capitalise,noabbrev]{cleveref}
\usepackage{xcolor}

\hypersetup{
  colorlinks=true,
  linkcolor=blue!55!black,
  citecolor=green!45!black,
  urlcolor=blue!55!black,
  pdftitle={A Sharp Curvature Threshold for GLMY Path Homology},
  pdfauthor={Shuliang Bai, Jingyan Li, and Shing-Tung Yau}
}

\newtheorem{theorem}{Theorem}[section]

\newaliascnt{proposition}{theorem}
\newtheorem{proposition}[proposition]{Proposition}
\aliascntresetthe{proposition}

\newaliascnt{lemma}{theorem}
\newtheorem{lemma}[lemma]{Lemma}
\aliascntresetthe{lemma}

\newaliascnt{corollary}{theorem}
\newtheorem{corollary}[corollary]{Corollary}
\aliascntresetthe{corollary}

\newaliascnt{conjecture}{theorem}

\aliascntresetthe{conjecture}

\newaliascnt{problem}{theorem}
\newtheorem{problem}[problem]{Problem}
\aliascntresetthe{problem}

\theoremstyle{definition}
\newaliascnt{definition}{theorem}
\newtheorem{definition}[definition]{Definition}
\aliascntresetthe{definition}

\newaliascnt{example}{theorem}

\aliascntresetthe{example}

\theoremstyle{remark}
\newaliascnt{remark}{theorem}
\newtheorem{remark}[remark]{Remark}
\aliascntresetthe{remark}

\crefname{theorem}{Theorem}{Theorems}
\crefname{proposition}{Proposition}{Propositions}
\crefname{lemma}{Lemma}{Lemmas}
\crefname{corollary}{Corollary}{Corollaries}
\crefname{conjecture}{Conjecture}{Conjectures}
\crefname{problem}{Problem}{Problems}
\crefname{definition}{Definition}{Definitions}
\crefname{example}{Example}{Examples}
\crefname{remark}{Remark}{Remarks}

\newcommand{\F}{\mathbb F}
\newcommand{\R}{\mathbb R}
\newcommand{\Z}{\mathbb Z}
\newcommand{\PathH}{H^{\mathrm{GLMY}}}
\newcommand{\cC}{\mathcal C}
\newcommand{\cR}{\mathcal R}
\newcommand{\cA}{\mathcal A}
\newcommand{\Om}{\Omega}
\newcommand{\diam}{\operatorname{diam}}

\newcommand{\Lip}{\operatorname{Lip}}
\newcommand{\supp}{\operatorname{supp}}
\newcommand{\av}{\operatorname{av}}

\newcommand{\Xshort}[1]{X_{\leq #1}}

\title[A Sharp Curvature Threshold for GLMY Path Homology]{A Sharp Curvature Threshold for GLMY Path Homology}
\author[Shuliang Bai, Jingyan Li, and Shing-Tung Yau]{%
Shuliang Bai$^{1}$\quad Jingyan Li$^{1,*}$\quad Shing-Tung Yau$^{1,2}$
\\[6pt]
$^{1}$Beijing Institute of Mathematical Sciences and Applications (BIMSA),
Beijing 101408, China
\\
$^{2}$Yau Mathematical Sciences Center, Tsinghua University,
Beijing 100084, China
\\
$^{*}$Corresponding author:
\href{mailto:jingyanli@bimsa.cn}{jingyanli@bimsa.cn}}
\date{}


\begin{document}

\begin{abstract}
Let $G$ be a finite simple graph with at least one edge.  We prove the sharp
vanishing theorem
\[
\kappa_{\min}^{\mathrm{LLY}}(G)>\frac12
\quad\Longrightarrow\quad
\PathH_1(G;\R)=0.
\]
Equivalently, nonzero first GLMY path homology forces an edge of
Lin--Lu--Yau curvature at most $1/2$.  The threshold $1/2$ is sharp and is
attained by $C_5$.  The proof combines the cycle-space description of first
GLMY path homology with the limit-free Laplacian characterization of
Lin--Lu--Yau curvature.  As a secondary consequence of the
curvature-preserving universal-cover method, we prove that if $G$ is connected
and $\kappa_{\min}^{\mathrm{LLY}}(G)>0$, then
$\pi_1(\Xshort{5}(G),o)$ is finite, where
$\Xshort{5}(G)$ is obtained by filling every simple cycle of length at most
five.  Equivalently, the normal subgroup generated by based simple
$5$-cycle loops has finite index in $\pi_1^{\mathrm{GLMY}}(G,o)$.  In higher
degrees the situation is different: for each integer $r\geq1$, the Cartesian
product $T_r=C_5^{\square r}$ has curvature $1/(2r)$ on every edge and, for
every field $\F$,
\[
\PathH_p(T_r;\F)\cong\F^{\binom rp}\qquad(0\leq p\leq r),
\]
so strict positivity of Lin--Lu--Yau curvature does not force
higher-dimensional GLMY path homology to vanish.
\end{abstract}

\maketitle

\section{Introduction}

GLMY path homology of digraphs was developed by Grigor'yan, Lin, Muranov, and
Yau; we use the formulation in their published account
\cite{GrigoryanLinMuranovYau2020}.  Its homotopy invariance was proved
in the homotopy theory for digraphs developed by the same authors
\cite{GrigoryanLinMuranovYau2014}.  Grigor'yan, Muranov, and Yau proved
K\"unneth formulas for GLMY path homology under Cartesian products and joins of digraphs
\cite{GrigoryanMuranovYau2017}.  We regard an undirected graph as the
bidirected digraph obtained by replacing each edge by the two opposite
arrows.  For such graphs, Kempton, M\"unch, and Yau identified first GLMY path
homology with the ordinary cycle space modulo the subspace generated
by simple $3$- and $4$-cycles \cite{KemptonMunchYau2021}.

Ollivier introduced coarse Ricci curvature for Markov chains on metric
spaces \cite{Ollivier2009}.  Lin, Lu, and Yau introduced the graph curvature
used here as a modification of Ollivier's construction
\cite{LinLuYau2011}.  Kempton, M\"unch, and Yau proved that strictly
positive Bakry--\'Emery curvature forces first GLMY path homology to vanish
\cite{KemptonMunchYau2021}; the analogous implication for positive
Lin--Lu--Yau curvature is false.  Indeed,
\[
\kappa_{\mathrm{LLY}}(C_5)=\frac12,
\qquad
\PathH_1(C_5;\R)\cong\R.
\]
Thus strict positivity alone is insufficient.  Our main result determines
the sharp Lin--Lu--Yau curvature threshold for the vanishing of first GLMY
path homology.

Recent work has established rigidity, classification, curvature formulas,
and order bounds for positively curved graphs under various structural
assumptions \cite{Hehl2025,Hehl2026,LinYou2026}.  Particularly
relevant here, Hehl and M\"unch obtained Betti-number estimates for a
curvature-adapted two-dimensional CW complex and developed a universal-cover
mechanism that preserves the local metric data entering the curvature
\cite{HehlMunch2025}.  We use that mechanism in
Section~\ref{sec:short} to derive a fundamental-group consequence for an
explicit short-cycle filling complex.  The sharp one-half theorem itself is
proved by a different argument, based on the interaction between short-cycle
relations and the variational formula for Lin--Lu--Yau curvature.

\begin{theorem}[Sharp one-half vanishing theorem]\label{thm:half-vanishing}
Let $G$ be a finite simple graph with at least one edge.  If
\[
\kappa_{\min}^{\mathrm{LLY}}(G)>\frac12,
\]
then
\[
\PathH_1(G;\R)=0.
\]
Equivalently, if $\PathH_1(G;\R)\neq0$, then there exists
$\{x,y\}\in E(G)$ such that
\[
\kappa_{\mathrm{LLY}}(x,y)\leq\frac12.
\]
The constant $1/2$ is optimal.
\end{theorem}

The formulation over $\R$ is convenient because the proof selects an
orthogonal representative in the real cycle space.  The conclusion in fact
holds over every field of characteristic zero.  Indeed, the incidence matrix
and the matrices of the $3$- and $4$-cycle relations have integer
entries, so their ranks over any characteristic-zero field agree with their
ranks over $\mathbb Q$ and $\R$.  Thus nonvanishing over a characteristic-zero field
would imply nonvanishing over $\R$, contradicting the theorem.  This argument
does not address positive characteristic, where the rank of the short-cycle
relation matrix may drop.

The proof selects a nonzero circulation in the cycle space that is orthogonal
to every $3$- and $4$-cycle vector.  At an edge on which
this circulation has maximal absolute value, the short-cycle relations define
a local potential on the union of the two endpoint neighborhoods.  After a
piecewise-linear contraction and a McShane extension, this potential is an
admissible test function for the Laplacian characterization of
Lin--Lu--Yau curvature.  The resulting estimate is at most $1/2$.

As a secondary topological consequence, let $\Xshort{5}(G)$ be the
two-dimensional CW complex obtained from $G$ by attaching one $2$-cell along
each simple cycle of length $3$, $4$, or $5$.  Specializing the local metric
preservation mechanism of Hehl and M\"unch \cite{HehlMunch2025} to this
explicit filling complex gives the following fundamental-group bound.  If
$G$ is finite and connected, $\kappa_{\min}(G)>0$, $n=|V(G)|$, $\Delta$ is
the maximum degree, and
\[
R=\left\lfloor\frac{2}{\kappa_{\min}(G)}\right\rfloor,
\]
where $\lfloor t\rfloor$ denotes the greatest integer not exceeding $t$, then,
for every base vertex $o$,
\[
|\pi_1(\Xshort{5}(G),o)|
\leq
\frac{1+\Delta\sum_{j=0}^{R-1}(\Delta-1)^j}{n}.
\]
In particular,
$\pi_1(\Xshort{5}(G),o)$ is finite.  Equivalently, for every base vertex $o$,
the quotient
\[
\pi_1^{\mathrm{GLMY}}(G,o)/N_5(G,o)
\]
is finite with the same bound.

Positive curvature alone does not imply vanishing in higher degree.  For each
integer $r\geq1$, let
\[
T_r=C_5^{\square r},
\]
the Cartesian product of $r$ copies of $C_5$.  Every edge of $T_r$ has
curvature $1/(2r)$, and the K\"unneth formula for GLMY path homology
gives
\[
\PathH_p(T_r;\F)\cong\F^{\binom rp}
\qquad(0\leq p\leq r).
\]
For every $p\geq1$, taking $r=p$ gives
\[
\PathH_p(T_p;\F)\cong\F\neq0,
\qquad
\kappa_{\min}^{\mathrm{LLY}}(T_p)=\frac{1}{2p}.
\]
Hence the extremal quantity
defined in \cref{sec:problems} satisfies
$K_p^{\mathrm{LLY}}(\F)\geq1/(2p)$.

The remainder of the paper is organized as follows.  Section~\ref{sec:prelim}
fixes notation and recalls the cycle-space description of first GLMY path homology.
Section~\ref{sec:half} proves the sharp one-half vanishing theorem.
Section~\ref{sec:short} records a short-cycle fundamental-group consequence
of the curvature-preserving universal-cover method.
Section~\ref{sec:examples-extremal} collects the sharp five-cycle, a positively
curved seven-vertex example with nonabelian GLMY fundamental group, and
Cartesian-product examples with nonzero higher-dimensional GLMY path homology, and
then formulates the resulting extremal curvature problem.

\section{Preliminaries}\label{sec:prelim}

Throughout, graphs are simple and locally finite.  Graphs whose homology is
computed are finite; covering graphs may be infinite.  Unless otherwise specified,
$\F$ is a field.  For an undirected graph $G$, an edge is an unordered pair
$\{x,y\}\in E(G)$; we write $x\sim y$ for adjacency.  Ordered pairs $(x,y)$
are used only for directed edges or for a chosen orientation of an undirected
edge.  The girth of an acyclic graph is taken to be $+\infty$.
We write a simple cycle with cyclically ordered vertices
$v_0,\ldots,v_{m-1}$ as $(v_0,\ldots,v_{m-1})$; the initial vertex is not
repeated.  By contrast, a closed walk is written with its initial vertex
repeated at the end.

\subsection{GLMY path homology}

Let $V$ be a finite set.  An elementary $p$-path is a symbol
\[
e_{i_0\ldots i_p},\qquad i_0,\ldots,i_p\in V.
\]
Let $\Lambda_p(V;\F)$ be the vector space spanned by all elementary $p$-paths.
Its boundary operator is defined on elementary paths by
\begin{equation}\label{eq:path-boundary}
\partial e_{i_0\ldots i_p}
=\sum_{q=0}^{p}(-1)^q e_{i_0\ldots\widehat{i_q}\ldots i_p}.
\end{equation}
An elementary path is \emph{regular} if $i_{q-1}\neq i_q$ for every
$q=1,\ldots,p$.  After declaring every nonregular elementary path to be zero,
one obtains the regular path space $\cR_p(V;\F)$; the operator \eqref{eq:path-boundary}
descends to
\[
\partial:\cR_p(V;\F)\longrightarrow\cR_{p-1}(V;\F).
\]

Let $D=(V,E_D)$ be a finite digraph, where $E_D\subseteq V\times V$ is its
set of directed edges.  A regular elementary path $e_{i_0\ldots i_p}$ is
\emph{allowed} if
\[
(i_{q-1},i_q)\in E_D
\qquad(1\leq q\leq p).
\]
Let $\cA_p(D;\F)\subseteq\cR_p(V;\F)$ be the span of the allowed elementary
$p$-paths.

\begin{definition}[GLMY path chain complex]
The space of boundary-invariant allowed $p$-paths is
\[
\Om_p(D;\F)
=\{u\in\cA_p(D;\F):\partial u\in\cA_{p-1}(D;\F)\}.
\]
Then $\partial\Om_p(D;\F)\subseteq\Om_{p-1}(D;\F)$, and the $p$-th GLMY path
homology group is
\[
\PathH_p(D;\F)
=\frac{\ker\bigl(\partial:\Om_p(D;\F)\to\Om_{p-1}(D;\F)\bigr)}
{\operatorname{im}\bigl(\partial:\Om_{p+1}(D;\F)\to\Om_p(D;\F)\bigr)}.
\]
Its dimension is denoted by $\beta_p^{\mathrm{GLMY}}(D;\F)$.
\end{definition}

For an undirected graph $G$, let $\overleftrightarrow{G}$ be the bidirected
digraph obtained by replacing each unordered edge $\{x,y\}\in E(G)$ by the
two directed edges $(x,y)$ and $(y,x)$.

\begin{definition}[GLMY path homology of an undirected graph]
We set
\[
\PathH_p(G;\F):=\PathH_p(\overleftrightarrow{G};\F).
\]
\end{definition}

We use the regular-path formulation and notation of Grigor'yan, Lin, Muranov,
and Yau \cite{GrigoryanLinMuranovYau2020}.

\subsection{Cycle spaces and first GLMY path homology}

Let $G$ be a finite undirected graph.  Define $C_0(G;\F)$ to be the vector
space with basis $\{e_v:v\in V(G)\}$.  For each unordered edge
$\{u,v\}\in E(G)$, choose exactly one of the two ordered pairs $(u,v)$ or
$(v,u)$, and denote the resulting set of chosen orientations by $E^+(G)$.
Thus $E^+(G)$ contains exactly one ordered pair for each edge of $G$.
For $(u,v)\in E^+(G)$, let $e_{uv}$ denote the corresponding elementary
$1$-path, and define
\[
C_1(G;\F)
:=\operatorname{span}_{\F}\{e_{uv}:(u,v)\in E^+(G)\}.
\]
The incidence boundary is
\[
\partial_G:C_1(G;\F)\longrightarrow C_0(G;\F),
\qquad
\partial_G e_{uv}=e_v-e_u.
\]
The cycle space of $G$ is
\[
\cC(G;\F):=\ker\partial_G.
\]
If $c(G)$ denotes the number of connected components of $G$, then
\[
\dim\cC(G;\F)=|E(G)|-|V(G)|+c(G).
\]

Let $C=(v_0,v_1,\ldots,v_{m-1})$ be a simple cycle with a chosen cyclic
orientation.  Its oriented cycle vector $\mathbf c_C\in\cC(G;\F)$ is obtained
by traversing $C$ in that orientation and, for each traversed edge
$v_i v_{i+1}$, using coefficient $+1$ on the chosen basis vector if
$(v_i,v_{i+1})\in E^+(G)$ and coefficient $-1$ on the chosen basis vector if
$(v_{i+1},v_i)\in E^+(G)$, where indices are taken modulo $m$.  Reversing the cyclic orientation replaces
$\mathbf c_C$ by $-\mathbf c_C$.

For $j\geq3$, let $\cC_j(G;\F)$ be the subspace of $\cC(G;\F)$ spanned by
the vectors $\mathbf c_C$ over all simple $j$-cycles $C$.  For $k\geq3$, set
\[
\cC_{\leq k}(G;\F):=\sum_{j=3}^{k}\cC_j(G;\F).
\]
Thus $\cC_{\leq k}(G;\F)$ is the subspace of the cycle space generated by
simple cycles of length at most $k$.

The following description of first GLMY path homology is due to Kempton,
M\"unch, and Yau \cite{KemptonMunchYau2021}.

\begin{lemma}[Cycle-space presentation of $\PathH_1$]\label{lem:H1-cycle-space}
For a finite undirected graph $G$ and a field $\F$, there is a natural
isomorphism
\[
\PathH_1(G;\F)
\cong
\cC(G;\F)/\cC_{\leq4}(G;\F).
\]
Equivalently, in degree one the regular path boundaries impose precisely the
cycle relations generated by simple $3$- and $4$-cycles.
\end{lemma}

\subsection{Short-cycle filling complexes}

For an integer $k\geq3$, define $\Xshort{k}(G)$ to be the two-dimensional
CW complex obtained from the one-skeleton $G$ by attaching one $2$-cell along
every simple cycle of length at most $k$.  Reversing the orientation of an
attaching map changes the corresponding cellular boundary by a sign and hence
does not change the subspace generated by the $2$-cell boundaries.

For later use, let $G$ be a connected graph with base vertex $o$.  We denote by
\[
\pi_1^{\mathrm{GLMY}}(G,o)
\]
the fundamental group of $G$ in the GLMY homotopy theory.  A based edge path
is an edge path that starts and ends at $o$, and
$\pi_1^{\mathrm{GLMY}}(G,o)$ is the group of equivalence classes of such
paths under the GLMY edge-path relations.  For graphs, the edge paths, their
local equivalence relations, and the resulting fundamental groupoid are given
in \cite[Definitions~5.2--5.3 and Theorem~5.4]{GrigoryanJimenezMuranov2018};
the vertex group at $o$ is naturally isomorphic to the graph fundamental group.
Equivalently, via the associated symmetric digraph, this is the
$C$-homotopy fundamental group introduced in
\cite[Theorem~4.20]{GrigoryanLinMuranovYau2014}.

The superscript $\mathrm{GLMY}$ is used only to distinguish this graph
fundamental group from the ordinary topological fundamental groups of the CW
complexes $\Xshort{k}(G)$.  Kempton, M\"unch, and Yau proved that the GLMY
fundamental group of an undirected graph is naturally isomorphic to the
ordinary fundamental group of the two-dimensional complex obtained by
filling every simple $3$- and $4$-cycle
\cite[Corollary~7.4]{KemptonMunchYau2021}.  In the notation of this paper,
\begin{equation}\label{eq:GLMY-pi1-X4}
\pi_1^{\mathrm{GLMY}}(G,o)
\cong
\pi_1(\Xshort{4}(G),o).
\end{equation}

We shall also use the relation between this fundamental group and first path
homology.  In the next formula, GLMY path homology is taken with integral
coefficients, using the same regular path chain construction over $\Z$.
For every connected graph, the GLMY abelianization theorem gives
\begin{equation}\label{eq:GLMY-abelianization}
\bigl(\pi_1^{\mathrm{GLMY}}(G,o)\bigr)_{\mathrm{ab}}
\cong
\PathH_1(G;\Z),
\end{equation}
where GLMY path homology is computed on the associated symmetric digraph
\cite[Theorem~4.23]{GrigoryanLinMuranovYau2014}.  The same relation also follows from
\cite[Proposition~7.7]{KemptonMunchYau2021} together with
\cref{lem:H1-cycle-space}.  Since the integral path chain groups are free
abelian, extension of scalars gives
\[
\PathH_1(G;\Z)\otimes_{\Z}\R
\cong
\PathH_1(G;\R),
\]
and hence
\[
\bigl(\pi_1^{\mathrm{GLMY}}(G,o)\bigr)_{\mathrm{ab}}
\otimes_{\Z}\R
\cong
\PathH_1(G;\R).
\]

Choose a spanning tree $T$ of $G$.  There is one standard generator for each
edge outside $T$: follow $T$ from $o$ to one endpoint, traverse the non-tree
edge, and return to $o$ along $T$.  Each $3$- or $4$-cycle $2$-cell
contributes the corresponding boundary word after the tree edges are set equal
to the identity.

For each simple $5$-cycle $C$, choose a vertex of $C$, one of the two cyclic
orientations of $C$, and the path in $T$ from $o$ to that vertex.  Follow the
tree path, traverse $C$ in the chosen orientation, and return along the inverse
tree path.  Let
\[
N_5(G,o)\triangleleft\pi_1^{\mathrm{GLMY}}(G,o)
\]
be the normal closure of these based $5$-cycle loops.  Reversing the cyclic
orientation replaces a generator by its inverse, while changing the chosen
vertex or the path from $o$ to $C$ conjugates the corresponding loop.  Hence
$N_5(G,o)$ is independent of all these choices.
Attaching the $5$-cycle $2$-cells gives
\begin{equation}\label{eq:pentagonal-normal-quotient}
\pi_1^{\mathrm{GLMY}}(G,o)/N_5(G,o)
\cong\pi_1(\Xshort{5}(G),o).
\end{equation}

\subsection{Lin--Lu--Yau curvature}

Let $G=(V,E)$ be a connected locally finite undirected graph, and let $d$ be
its graph distance.  For a vertex $x\in V$ of positive degree, define its open
and closed neighborhoods by
\[
N(x):=\{z\in V:z\sim x\},
\qquad
N[x]:=N(x)\cup\{x\},
\]
and write $d_x:=|N(x)|$ for the degree of $x$.

For an idleness parameter $\alpha\in[0,1]$, the lazy random-walk probability
measure centered at $x$ is the function $\mu_x^\alpha:V\to[0,1]$ defined by
\begin{equation}\label{eq:lazy-measure}
\mu_x^\alpha(z)=
\begin{cases}
\alpha,&z=x,\\
(1-\alpha)/d_x,&z\sim x,\\
0,&\text{otherwise}.
\end{cases}
\end{equation}
Thus $\mu_x^\alpha$ places mass $\alpha$ at $x$ and distributes the remaining
mass uniformly among the neighbors of $x$.

To compare the measures centered at two vertices, first recall the
$1$-Wasserstein distance for arbitrary finitely supported probability measures
on $V$.  Such a probability measure is a function $\mu:V\to[0,1]$ with
$\sum_{u\in V}\mu(u)=1$ and finite support, where
\[
\supp\mu:=\{u\in V:\mu(u)\neq0\}.
\]  If $\mu$ and $\nu$ are two such
measures, a \emph{coupling} of $\mu$ and $\nu$ is a function
$\pi:V\times V\to[0,\infty)$ satisfying
\[
\sum_{v\in V}\pi(u,v)=\mu(u)\quad(u\in V),
\qquad
\sum_{u\in V}\pi(u,v)=\nu(v)\quad(v\in V).
\]
The number $\pi(u,v)$ is the mass transported from $u$ to $v$.  Its total
transportation cost is $\sum_{u,v}d(u,v)\pi(u,v)$, and the minimum such cost is
\[
W_1(\mu,\nu)
:=\inf_{\pi}\sum_{u,v\in V}d(u,v)\pi(u,v),
\]
where the infimum is taken over all couplings of $\mu$ and $\nu$.  In the
curvature definitions below, the two measures are $\mu_x^\alpha$ and
$\mu_y^\alpha$.

Define once and for all
\[
\Lip_1(G)
:=\bigl\{\varphi:V\to\R:
|\varphi(u)-\varphi(v)|\leq d(u,v)
\text{ for all }u,v\in V\bigr\}.
\]
Thus $\Lip_1(G)$ is the class of real-valued $1$-Lipschitz functions on the
metric space $(V,d)$.  Kantorovich duality \cite[Chapter~2]{Villani2003} gives
\[
W_1(\mu,\nu)
=\sup_{\varphi\in\Lip_1(G)}
\sum_{z\in V}\varphi(z)\bigl(\mu(z)-\nu(z)\bigr).
\]

Following Ollivier \cite{Ollivier2009}, for distinct vertices $x,y$ define
\[
\kappa_\alpha(x,y)
:=1-\frac{W_1(\mu_x^\alpha,\mu_y^\alpha)}{d(x,y)}.
\]
Lin, Lu, and Yau proved that the following limit exists
\cite{LinLuYau2011}.

\begin{definition}[Lin--Lu--Yau curvature]
For distinct vertices $x,y$, define
\[
\kappa(x,y)
:=\lim_{\alpha\to1^-}
\frac{\kappa_\alpha(x,y)}{1-\alpha}.
\]
If $G$ is finite and has at least one edge, set
\[
\kappa_{\min}(G)
:=\min_{\{x,y\}\in E(G)}\kappa(x,y).
\]
We also write
\[
\kappa_{\mathrm{LLY}}(x,y):=\kappa(x,y),
\qquad
\kappa_{\min}^{\mathrm{LLY}}(G):=\kappa_{\min}(G),
\]
when it is useful to emphasize the curvature under consideration.
\end{definition}

We use the limit definition and the limit-free Laplacian formula below.

The following two facts will be used below.  The first is the locality of the
finite optimal-transport problem, and the second is the Lin--Lu--Yau
Bonnet--Myers estimate \cite{LinLuYau2011}.

\begin{lemma}[Locality and Bonnet--Myers]\label{lem:LLY-basic}
Let $G$ be connected and locally finite.
\begin{enumerate}[label=\textup{(\roman*)}]
\item Let $x,y\in V$ have positive degree.  For $\alpha\in[0,1]$, the value
$W_1(\mu_x^\alpha,\mu_y^\alpha)$ depends only on the pairwise graph distances
among vertices in
\[
\supp\mu_x^\alpha\cup\supp\mu_y^\alpha
\subseteq N[x]\cup N[y].
\]
For $\alpha\in(0,1)$, one has $\supp\mu_x^\alpha=N[x]$ and
$\supp\mu_y^\alpha=N[y]$.
\item If $\kappa(x,y)\geq\kappa_0>0$ for every edge $\{x,y\}\in E(G)$, then
\[
\diam(G)\leq\frac{2}{\kappa_0},
\qquad
\diam(G):=\sup_{u,v\in V}d(u,v).
\]
\end{enumerate}
\end{lemma}

For the remainder of this subsection, assume that $G$ is finite, connected,
and has at least one edge.  If $\varphi:V\to\R$ is any real-valued function,
define its neighbor average and normalized graph Laplacian at $x\in V$ by
\[
\av_x\varphi
:=\frac1{d_x}\sum_{u\sim x}\varphi(u),
\qquad
\Delta\varphi(x)
:=\frac1{d_x}\sum_{u\sim x}\bigl(\varphi(u)-\varphi(x)\bigr)
=\av_x\varphi-\varphi(x).
\]
Because $G$ is connected and has an edge, every vertex has positive degree,
so these expressions are defined at every vertex.

Let $x\sim y$.  The limit-free Laplacian characterization of M\"unch and
Wojciechowski \cite[Theorem~2.1]{MunchWojciechowski2019} gives, for the
normalized Laplacian,
\begin{equation}\label{eq:curvature-laplacian}
\kappa(x,y)
=\inf_{\substack{h\in\Lip_1(G)\\ h(y)-h(x)=1}}
\bigl(\Delta h(x)-\Delta h(y)\bigr).
\end{equation}
Here the infimum is taken over the class $\Lip_1(G)$ defined above.  Since
$x\sim y$, the graph distance between the two endpoints is $d(x,y)=1$.
Moreover, because $G$ is finite, every real-valued function on $V(G)$ has
finite support.  Hence the compact-support condition in
\cite[Theorem~2.1]{MunchWojciechowski2019} is automatically satisfied.

Adding a constant to a test function $h$ changes neither its Lipschitz
constant nor its Laplacian.  Hence the infimum may be restricted to functions
satisfying $h(x)=0$ and $h(y)=1$.  For such an $h$, set $g:=1-h$.  Then
$g\in\Lip_1(G)$, $g(x)=1$, and $g(y)=0$.  We therefore define the admissible
set for the equivalent maximization problem by
\[
\mathcal F_{x,y}
:=\{g\in\Lip_1(G): g(x)=1,\ g(y)=0\}.
\]
Since $\Delta h(v)=\av_v h-h(v)$,
\[
\Delta h(x)-\Delta h(y)
=1-\bigl(\av_xg-\av_yg\bigr).
\]
Consequently,
\begin{equation}\label{eq:curvature-LP}
\kappa(x,y)
=1-\max_{g\in\mathcal F_{x,y}}
\bigl(\av_xg-\av_yg\bigr).
\end{equation}
By the definition of $\Lip_1(G)$, every $g\in\mathcal F_{x,y}$ satisfies
\[
|g(u)-g(v)|\leq d(u,v)
\qquad (u,v\in V(G)).
\]
In particular, since $g(x)=1$,
\[
|g(v)-1|\leq d(v,x)
\qquad (v\in V(G)).
\]
Because $G$ is finite and connected, the numbers $d(v,x)$ are finite and
uniformly bounded as $v$ ranges over $V(G)$.  Hence
$\mathcal F_{x,y}$ is bounded in $\R^{V(G)}$.  It is also closed, since it is
defined by the linear inequalities above together with the affine conditions
$g(x)=1$ and $g(y)=0$.  Therefore $\mathcal F_{x,y}$ is a compact polytope,
and the continuous linear functional
$g\mapsto\av_xg-\av_yg$ attains its maximum on $\mathcal F_{x,y}$.

\section{The sharp one-half vanishing theorem}\label{sec:half}

This section proves \cref{thm:half-vanishing}.  The proof uses the
cycle-space quotient in \cref{lem:H1-cycle-space} and the limit-free
Laplacian formula \eqref{eq:curvature-laplacian}.

\begin{proof}[Proof of \cref{thm:half-vanishing}]
First GLMY path homology is the direct sum of the first GLMY path homology groups of
the connected components, and the curvature of an edge is determined within
its connected component.  Replacing $G$ by a connected component with nonzero
first GLMY path homology reduces the proof to the case in which $G$ is connected.
Use the fixed orientation $E^+(G)$ from \cref{sec:prelim} and give
$C_1(G;\R)$ the Euclidean inner product for which the $1$-paths
$e_{uv}$, $(u,v)\in E^+(G)$, form an orthonormal basis.  Let
\[
Z_1=\ker(\partial_G:C_1(G;\R)\longrightarrow C_0(G;\R))
\]
be the cycle space, and put
\[
B_{\leq4}:=\cC_{\leq4}(G;\R)
=\operatorname{span}_{\R}
\{\mathbf c_C:C\text{ is a simple $3$- or $4$-cycle}\}.
\]
Here $\mathbf c_C$ is the oriented cycle vector defined in
\cref{sec:prelim}; changing the cyclic orientation of $C$ changes
$\mathbf c_C$ only by a sign.  By \cref{lem:H1-cycle-space},
\[
\PathH_1(G;\R)\cong Z_1/B_{\leq4}.
\]
Since $B_{\leq4}\subseteq Z_1$ and the quotient is nonzero, the orthogonal
complement of $B_{\leq4}$ inside $Z_1$ is nonzero.  Choose
\[
0\neq\omega\in Z_1\cap B_{\leq4}^{\perp}.
\]
Recall that
\[
C_1(G;\R)
=
\operatorname{span}_{\R}
\{e_{uv}:(u,v)\in E^+(G)\},
\]
where $E^+(G)$ contains exactly one chosen orientation of each unordered edge
of $G$.  Thus
\[
\{e_{uv}:(u,v)\in E^+(G)\}
\]
is the fixed orthonormal basis of $C_1(G;\R)$ used above.  Write
\[
\omega
=
\sum_{(u,v)\in E^+(G)}\omega_{uv}e_{uv}.
\]
For the reverse orientation of the same unordered edge, define
\[
\omega_{vu}:=-\omega_{uv}.
\]
Hence $\omega_{uv}$ is defined for every ordered adjacent pair $u\sim v$ and
satisfies
\[
\omega_{uv}=-\omega_{vu}.
\]

We now record separately the consequences of the two properties
\[
\omega\in Z_1
\qquad\text{and}\qquad
\omega\in B_{\leq4}^{\perp}.
\]
Since $Z_1=\ker\partial_G$ and $\omega\in Z_1$, we have
$\partial_G\omega=0$.  With the above antisymmetric convention, the
coefficient of $e_v$ in $\partial_G\omega$ is
$-\sum_{z\sim v}\omega_{vz}$.  Therefore
\begin{equation}\label{eq:omega-divergence}
\sum_{z\sim v}\omega_{vz}=0
\qquad(v\in V(G)).
\end{equation}

On the other hand,
\[
B_{\leq4}
=
\operatorname{span}_{\R}
\{\mathbf c_C:C\text{ is a simple $3$- or $4$-cycle}\},
\]
and $\omega\in B_{\leq4}^{\perp}$.  Hence
\[
\langle\omega,\mathbf c_C\rangle=0
\]
for every simple $3$- or $4$-cycle $C$.  If
\[
C=(v_0,v_1,\ldots,v_{k-1}),
\qquad k\in\{3,4\},
\]
is given a cyclic orientation, then the definition of the oriented cycle
vector $\mathbf c_C$ gives
\begin{equation}\label{eq:omega-short-closed}
\sum_{i=0}^{k-1}\omega_{v_i v_{i+1}}=0,
\end{equation}
where the indices are taken modulo $k$.  Thus, for a simple $3$- or $4$-cycle
equipped with a cyclic orientation, the sum of the
coefficients $\omega_{uv}$ along its oriented boundary is zero.

Normalize $\omega$ so that
\[
\max_{u\sim v}|\omega_{uv}|=1,
\]
and choose an edge $\{x,y\}\in E(G)$, labeling its endpoints so that
\begin{equation}\label{eq:omega-max-edge}
\omega_{xy}=1.
\end{equation}

\medskip
\noindent\emph{Local potential on the endpoint neighborhoods.}
The curvature estimate will be tested at the edge $\{x,y\}$.  The quantities
$\Delta h(x)$ and $\Delta h(y)$ depend only on the values of $h$ on the closed
neighborhoods $N[x]$ and $N[y]$, respectively.  Since $x\sim y$, both closed
neighborhoods are contained in
\[
S:=N(x)\cup N(y).
\]
We shall construct a scalar function $F:S\to\R$ whose discrete differences
recover the coefficients $\omega_{uv}$ on every edge with both endpoints in
$S$.  Relation \eqref{eq:omega-short-closed} is the ingredient that makes this
local integration possible.

To define $F$ without overlapping cases, set
\[
N_{xy}:=N(x)\cap N(y),
\]
\[
A_x:=N(x)\setminus\bigl(N(y)\cup\{y\}\bigr),
\qquad
A_y:=N(y)\setminus\bigl(N(x)\cup\{x\}\bigr).
\]
Then
\[
N(x)=\{y\}\sqcup A_x\sqcup N_{xy},
\qquad
N(y)=\{x\}\sqcup A_y\sqcup N_{xy},
\]
and therefore
\[
S=\{x,y\}\sqcup A_x\sqcup N_{xy}\sqcup A_y.
\]
Define $F:S\to\R$ by
\begin{equation}\label{eq:local-potential}
F(z):=
\begin{cases}
0,&z=x,\\
1,&z=y,\\
\omega_{xz},&z\in A_x\cup N_{xy},\\
1+\omega_{yz},&z\in A_y.
\end{cases}
\end{equation}
The values at $x$ and $y$ satisfy
\[
F(y)-F(x)=1=\omega_{xy},
\qquad
F(x)-F(y)=-1=\omega_{yx}.
\]
For a common neighbor $z\in N_{xy}$, relation
\eqref{eq:omega-short-closed} applied to the simple $3$-cycle
$(x,y,z)$ gives
\[
0=\omega_{xy}+\omega_{yz}+\omega_{zx}
  =1+\omega_{yz}-\omega_{xz}.
\]
Hence
\begin{equation}\label{eq:common-neighbor-potential}
\omega_{xz}=1+\omega_{yz}
\qquad(z\in N_{xy}).
\end{equation}
Thus the value assigned to a common neighbor by
\eqref{eq:local-potential} is also the value obtained by integrating from the
$y$-side.

We next prove that $F$ has the required edge differences.  Namely, for every
ordered adjacent pair $(u,v)$ with $u,v\in S$,
\begin{equation}\label{eq:potential-edge-difference}
F(v)-F(u)=\omega_{uv}.
\end{equation}
Since $G$ is simple, $u\neq v$.  We distinguish the following cases.

\smallskip
\noindent\emph{Case 1: one endpoint is $x$ or $y$.}
Suppose first that $u=x$.  If $v=y$, then
\eqref{eq:potential-edge-difference} is exactly
$F(y)-F(x)=\omega_{xy}=1$.  If $v\in A_x\cup N_{xy}$, then the definition of
$F$ gives
\[
F(v)-F(x)=\omega_{xv}.
\]
Similarly, if $u=y$ and $v\in A_y$, then
\[
F(v)-F(y)=\omega_{yv},
\]
while for $v\in N_{xy}$ the same identity follows from
\eqref{eq:common-neighbor-potential}.  The case $v=x$ or $v=y$ follows by
reversing the ordered pair and using
$\omega_{vu}=-\omega_{uv}$.

\smallskip
\noindent\emph{Case 2: the two endpoints lie in the same endpoint
neighborhood.}
Assume that neither endpoint is $x$ or $y$.  If $u,v\in N(x)$, then
$(x,u,v)$ is a simple $3$-cycle.  Applying
\eqref{eq:omega-short-closed} to its cyclic orientation gives
\[
0=\omega_{xu}+\omega_{uv}+\omega_{vx},
\]
and hence
\[
\omega_{uv}=\omega_{xv}-\omega_{xu}=F(v)-F(u).
\]
The same argument with the $3$-cycle $(y,u,v)$ proves
\eqref{eq:potential-edge-difference} when $u,v\in N(y)$.

\smallskip
\noindent\emph{Case 3: a cross edge.}
The only remaining possibility, after interchanging $u$ and $v$ if necessary,
is
\[
u\in A_x,
\qquad
v\in A_y.
\]
Then $(x,u,v,y)$ is a simple $4$-cycle.  Applying
\eqref{eq:omega-short-closed} to this cyclic orientation gives
\[
0=\omega_{xu}+\omega_{uv}+\omega_{vy}+\omega_{yx}.
\]
Using $\omega_{xy}=1$ and antisymmetry,
\[
\omega_{uv}=1+\omega_{yv}-\omega_{xu}.
\]
By \eqref{eq:local-potential},
$F(u)=\omega_{xu}$ and $F(v)=1+\omega_{yv}$, and therefore
\[
\omega_{uv}=F(v)-F(u).
\]
For the reversed ordered pair, the identity follows from antisymmetry.

These cases exhaust all ordered adjacent pairs in $S$, and hence
\eqref{eq:potential-edge-difference} holds throughout the induced local
configuration.

The function $F$ now has the correct edge differences, but it need not yet be
$1$-Lipschitz with respect to the ambient graph distance on every pair of
vertices in $S$.  Indeed, the normalization $|\omega_{uv}|\leq1$ implies
\[
F(S)\subseteq[-1,2],
\]
whose diameter is $3$.  We therefore compress only the values outside the
interval $[0,1]$, while keeping the normalization $F(x)=0$ and $F(y)=1$
unchanged.  Besides making the local function $1$-Lipschitz, this compression
halves precisely the negative contributions at $x$ and the positive
contributions at $y$; this is the source of the factor $1/2$ in the final
curvature estimate.  Define the nondecreasing piecewise-linear map
$T:[-1,2]\to[-1/2,3/2]$ by
\[
T(t)=
\begin{cases}
t/2,&-1\leq t\leq0,\\
t,&0\leq t\leq1,\\
(t+1)/2,&1\leq t\leq2.
\end{cases}
\]
Since $T$ is piecewise linear with absolute slope at most $1$, it is
$1$-Lipschitz.  Moreover, $T(0)=0$ and $T(1)=1$.  Define
\[
f:=T\circ F:S\to\R.
\]
Then
\begin{equation}\label{eq:f-local-values}
f(x)=0,\qquad f(y)=1,\qquad
f(S)\subseteq[-1/2,3/2].
\end{equation}
If $u,v\in S$ are adjacent, then the $1$-Lipschitz property of $T$ and
\eqref{eq:potential-edge-difference} give
\[
|f(u)-f(v)|\leq |F(u)-F(v)|
=|\omega_{uv}|\leq1.
\]
If $u=v$, the Lipschitz inequality is trivial.  If $u\neq v$
and $u,v\in S$ are not adjacent, then $d_G(u,v)\geq2$, while
\eqref{eq:f-local-values} gives $|f(u)-f(v)|\leq2$.  Therefore
\begin{equation}\label{eq:f-environment-lipschitz}
|f(u)-f(v)|\leq d_G(u,v)
\qquad(u,v\in S).
\end{equation}
Here and below, distances between vertices of $S$ are measured in the ambient
graph $G$, rather than in the induced subgraph on $S$.

The function $f$ is now $1$-Lipschitz on the subset $S$, but the variational
curvature formula \eqref{eq:curvature-laplacian} requires a function defined on
all of $V(G)$.  We therefore extend $f$ to the whole graph by the McShane
formula \cite{McShane1934}
\begin{equation}\label{eq:mcshane-extension}
\widetilde f(v)
=\inf_{s\in S}\{f(s)+d_G(v,s)\}.
\end{equation}
Equation \eqref{eq:f-environment-lipschitz} implies
\[
\widetilde f|_S=f,\qquad \widetilde f\in\Lip_1(G).
\]
In particular, $\widetilde f(x)=0$ and $\widetilde f(y)=1$.  Hence
\eqref{eq:curvature-laplacian} gives
\begin{equation}\label{eq:kappa-test-f}
\kappa_{\mathrm{LLY}}(x,y)
\leq\Delta\widetilde f(x)-\Delta\widetilde f(y).
\end{equation}
Because every neighbor of $x$ and $y$ lies in $S$ and
$\widetilde f|_S=f$, these two Laplacians are computed entirely from the
local values of $f$ constructed above.

For every $z\in N(x)$, we have $z\in S$ and hence
$\widetilde f(z)=f(z)$.  Using the definition of $F$ on $N(x)$ and the
piecewise definition of $T$, we obtain
\[
\widetilde f(z)-\widetilde f(x)
=f(z)-f(x)
=
\begin{cases}
\omega_{xz},&\omega_{xz}\geq0,\\
\omega_{xz}/2,&\omega_{xz}\leq0.
\end{cases}
\qquad(z\in N(x)).
\]
Therefore
\[
\begin{aligned}
\Delta\widetilde f(x)
&=\frac1{d_x}
  \left(
  \sum_{\substack{z\sim x\\\omega_{xz}>0}}\omega_{xz}
  +\frac12
  \sum_{\substack{z\sim x\\\omega_{xz}<0}}\omega_{xz}
  \right).
\end{aligned}
\]
Since \eqref{eq:omega-divergence} gives
\[
\sum_{\substack{z\sim x\\\omega_{xz}<0}}\omega_{xz}
=-\sum_{\substack{z\sim x\\\omega_{xz}>0}}\omega_{xz},
\]
it follows that
\begin{equation}\label{eq:laplacian-x-half}
\Delta\widetilde f(x)
=\frac1{2d_x}
\sum_{\substack{z\sim x\\\omega_{xz}>0}}\omega_{xz}.
\end{equation}

Similarly, for every $z\in N(y)$,
\[
\widetilde f(z)-\widetilde f(y)
=f(z)-f(y)
=
\begin{cases}
\omega_{yz},&\omega_{yz}\leq0,\\
\omega_{yz}/2,&\omega_{yz}\geq0,
\end{cases}
\qquad(z\in N(y)).
\]
Thus
\[
\Delta\widetilde f(y)
=\frac1{d_y}
  \left(
  \sum_{\substack{z\sim y\\\omega_{yz}<0}}\omega_{yz}
  +\frac12
  \sum_{\substack{z\sim y\\\omega_{yz}>0}}\omega_{yz}
  \right).
\]
Using \eqref{eq:omega-divergence} again gives
\begin{equation}\label{eq:laplacian-y-half}
\Delta\widetilde f(y)
=-\frac1{2d_y}
\sum_{\substack{z\sim y\\\omega_{yz}>0}}\omega_{yz}.
\end{equation}
Combining \eqref{eq:kappa-test-f}, \eqref{eq:laplacian-x-half}, and
\eqref{eq:laplacian-y-half} gives
\begin{equation}\label{eq:kappa-positive-flow-bound}
\kappa_{\mathrm{LLY}}(x,y)
\leq\frac12\left(
\frac1{d_x}\sum_{\substack{z\sim x\\\omega_{xz}>0}}\omega_{xz}
+
\frac1{d_y}\sum_{\substack{z\sim y\\\omega_{yz}>0}}\omega_{yz}
\right).
\end{equation}

It remains to bound each positive sum.  Fix $v\in\{x,y\}$.  Since
$|\omega_{vz}|\leq1$ for every $z\sim v$,
\[
\sum_{\substack{z\sim v\\\omega_{vz}>0}}\omega_{vz}
\leq
\bigl|\{z\sim v:\omega_{vz}>0\}\bigr|.
\]
On the other hand, \eqref{eq:omega-divergence} implies
\[
\sum_{\substack{z\sim v\\\omega_{vz}>0}}\omega_{vz}
=-\sum_{\substack{z\sim v\\\omega_{vz}<0}}\omega_{vz}
\leq
\bigl|\{z\sim v:\omega_{vz}<0\}\bigr|.
\]
Adding these inequalities gives
\[
2\sum_{\substack{z\sim v\\\omega_{vz}>0}}\omega_{vz}
\leq
\bigl|\{z\sim v:\omega_{vz}>0\}\bigr|
+
\bigl|\{z\sim v:\omega_{vz}<0\}\bigr|
\leq d_v.
\]
Therefore, for $v=x$ and $v=y$,
\begin{equation}\label{eq:positive-sum-half}
\frac1{d_v}
\sum_{\substack{z\sim v\\\omega_{vz}>0}}\omega_{vz}
\leq\frac12.
\end{equation}
Substituting \eqref{eq:positive-sum-half} into
\eqref{eq:kappa-positive-flow-bound} yields
\[
\kappa_{\mathrm{LLY}}(x,y)
\leq\frac12\left(\frac12+\frac12\right)
=\frac12.
\]

Taking the contrapositive proves the vanishing assertion.  Finally, $C_5$
has nonzero first GLMY path homology and curvature $1/2$ on every edge, as
computed in \cref{sec:examples}; hence the threshold and the strict
inequality are sharp.
\end{proof}

\begin{remark}[Necessary conditions at the threshold]
The proof also records restrictions on equality in the estimate above.  At
each of the two selected endpoints, equality in
\eqref{eq:positive-sum-half} requires the positive and negative supports of
the normalized circulation to exhaust the incident edges, to have the same
cardinality, and to carry coefficients of absolute value one.  In addition,
the compressed McShane extension must minimize the variational formula
\eqref{eq:curvature-laplacian}.  We do not attempt here to classify all graphs
for which these necessary conditions are simultaneously attained.
\end{remark}

\section{A short-cycle fundamental-group consequence of positive curvature}\label{sec:short}

This section records a fundamental-group consequence of the
curvature-preserving universal-cover method developed by Hehl and M\"unch
\cite{HehlMunch2025}.  Their curvature-adapted complex $M_2(G)$ is designed so
that its universal cover preserves the local metric information entering the
Ollivier-curvature problem; the local distance statement underlying this
mechanism is \cite[Lemma~2.3]{HehlMunch2025}.  We do not introduce a new
curvature-preserving cover here.  Instead, in the unweighted setting we use
the explicit complex $\Xshort{5}(G)$ and isolate a simple sufficient
condition---$5$-cycle preservation---that makes the same local metric
argument transparent.

The point we retain is a fundamental-group conclusion.  If
$\kappa_{\min}(G)>0$, local curvature preservation allows the
Lin--Lu--Yau Bonnet--Myers estimate to be applied to the one-skeleton of the
universal cover of $\Xshort{5}(G)$.  The resulting diameter bound, together
with the degree bound inherited from the finite base graph, forces the
universal cover to be finite and gives an explicit bound on
$|\pi_1(\Xshort{5}(G),o)|$ for any base vertex $o$.  Hehl and M\"unch already
develop the corresponding
Betti-number consequences for their curvature-adapted complex, so we do not
repeat analogous homological corollaries in this section.

\subsection{Local transport under short-cycle-preserving covers}

We use the standard graph-theoretic terminology for walks, with the following
conventions.

For an integer $r\geq0$, a \emph{walk of length $r$} in a graph $G$ is a
finite vertex sequence
\[
W=v_0v_1\cdots v_r
\]
such that
\[
\{v_{i-1},v_i\}\in E(G)
\qquad(1\le i\le r).
\]
Vertices and edges may occur more than once at different positions in the
walk.  Since $G$ is simple, consecutive vertices are distinct.  The walk is
\emph{closed} if $v_r=v_0$.  Its reverse is
\[
W^{-1}=v_rv_{r-1}\cdots v_0.
\]
Whenever the terminal vertex of one walk is the initial vertex of another,
we denote their concatenation by juxtaposition.  Thus expressions such as
$PW^{-1}$ below denote concatenations of walks, not products of algebraic
objects.

\begin{definition}[Graph covers and short-cycle preservation]
\label{def:graph-cover}
Let $G$ and $\widetilde G$ be connected simple graphs.  A \emph{graph cover}
$q:\widetilde G\to G$ is a surjective map on vertex sets such that, for every
$\widetilde v\in V(\widetilde G)$, the restriction
\[
q:N(\widetilde v)\longrightarrow N(q(\widetilde v))
\]
is a bijection.  For an integer $k\ge3$, the graph cover $q$ is
\emph{$k$-cycle preserving} if every lift of every simple cycle of $G$ of
length at most $k$ is closed.
\end{definition}

This is the standard locally bijective definition of a graph covering; see
\cite[Section~2.3]{KemptonMunchYau2021}.  Our $k$-cycle-preserving condition is
a special case of the $B$-preserving coverings considered there, with $B$ the
set of simple cycles of length at most $k$.

The local bijection in \cref{def:graph-cover} gives unique walk lifting.  More
precisely, if
\[
W=v_0v_1\cdots v_r
\]
is a walk in $G$ and $\widetilde v_0\in q^{-1}(v_0)$ is prescribed, the local
bijection property determines a unique walk
\[
\widetilde W=\widetilde v_0\widetilde v_1\cdots\widetilde v_r
\]
satisfying
\[
q(\widetilde v_i)=v_i
\qquad(0\le i\le r).
\]
We call $\widetilde W$ the \emph{lift of $W$ starting at $\widetilde v_0$}.

The following is the standard lifting property of CW coverings, applied to
$\Xshort{k}(G)$.

\begin{lemma}[One-skeleta of short-cycle covers]\label{lem:CW-graph-cover}
Let
\[
p:\widetilde X\longrightarrow \Xshort{k}(G)
\]
be a connected covering of CW complexes, and give $\widetilde X$ the lifted
CW structure.  Then the restriction
\[
p|_{\widetilde X^{(1)}}:\widetilde X^{(1)}\longrightarrow G
\]
is a connected $k$-cycle-preserving graph cover.
\end{lemma}

\begin{proof}
Evenly covered open stars of vertices show that every edge incident to a
vertex of $G$ has a unique incident lift at each vertex above it.  Hence the
restriction to the one-skeleta induces the neighborhood bijections required
in \cref{def:graph-cover}; it is surjective because $p$ is a covering, and it
is connected because $\widetilde X$ is path connected.

Let $C$ be a simple cycle of length at most $k$, and prescribe a lift of its
initial vertex.  The cycle $C$ is the attaching loop of a $2$-cell of
$\Xshort{k}(G)$.  By the lifting theorem, its characteristic map from the
disk lifts after prescribing a point above a boundary point.  The boundary of
the lifted disk is therefore a closed lift of $C$ with the prescribed initial
vertex.  By the unique walk lifting that follows from
\cref{def:graph-cover}, this boundary walk is the prescribed lift of $C$.
Thus every such lift is closed.
\end{proof}

We first extend $k$-cycle preservation from simple cycles to arbitrary closed
walks of length at most $k$.

\begin{lemma}\label{lem:short-closed-walk}
If a graph cover is $k$-cycle preserving, then every lift of every closed walk
of length at most $k$ is closed.
\end{lemma}

\begin{proof}
We argue by induction on the length of the closed walk.  The assertion is
trivial for length zero, and it is exactly the definition of
$k$-cycle preservation when the walk is a simple cycle.

Suppose that the closed walk $W$ is not simple.  If $W=A\,aba\,B$ contains an
immediate reversal, then $AB$ is a shorter closed walk.  By induction the
lift of $AB$ from any chosen lift of its initial vertex is closed.  The lift
of the two-edge segment $aba$, inserted after the lift of $A$, returns to the
same lifted vertex, so the lift of $W$ is closed as well.

Assume now that there is no immediate reversal.  Choose two occurrences of a
repeated vertex before the final return and write
\[
W=A\,C\,B,
\]
where $C$ is the nonempty closed subwalk between those occurrences.  Both
$C$ and $AB$ are strictly shorter closed walks.  Lift $A$ from the prescribed
initial lift.  By induction, the lift of $C$ from the endpoint of the lifted
$A$ returns to that same endpoint.  Continuing with the lift of $B$ then has
the same endpoint as the closed lift of $AB$, and hence returns to the
initial vertex.  Thus the lift of $W$ is closed.
\end{proof}

We now identify the metric information needed for the transport problem.
Let $q:\widetilde G\to G$ be a $5$-cycle-preserving graph cover and let
$\{\widetilde x,\widetilde y\}\in E(\widetilde G)$.  Write
\[
x=q(\widetilde x),\qquad y=q(\widetilde y).
\]
The lazy random-walk measures at $\widetilde x$ and $\widetilde y$ are
supported on $N[\widetilde x]$ and $N[\widetilde y]$, respectively.
Consequently, preservation of the Lin--Lu--Yau transport problem reduces to
preservation of the distances between these two closed neighborhoods.  The local metric mechanism is the one used by Hehl and M\"unch
\cite[Lemma~2.3]{HehlMunch2025}.  We include the short proof in our explicit
$5$-cycle-preserving formulation because it shows directly why cycles of
length at most five are sufficient.

\begin{lemma}[Local distance preservation]\label{lem:local-distance}
Under the notation above,
\[
d_{\widetilde G}(\widetilde u,\widetilde v)
=
d_G(q(\widetilde u),q(\widetilde v))
\]
for every
\[
\widetilde u\in N[\widetilde x],
\qquad
\widetilde v\in N[\widetilde y].
\]
\end{lemma}

\begin{proof}
Set
\[
u=q(\widetilde u),\qquad v=q(\widetilde v).
\]
Every walk in $\widetilde G$ projects to a walk of the same length in $G$.
In particular, a shortest walk in $\widetilde G$ projects to a walk in $G$,
so
\[
d_G(u,v)\le d_{\widetilde G}(\widetilde u,\widetilde v).
\]

For the reverse inequality, join $\widetilde u$ to $\widetilde v$ by the walk
obtained from the vertex sequence
\[
\widetilde u\,\widetilde x\,\widetilde y\,\widetilde v
\]
after deleting consecutive repetitions.  Its length is at most three.  Let
$W$ be its projection to $G$, and let $P$ be a shortest walk from $u$ to $v$,
so that $\ell(P)=d_G(u,v)$.  Then
\[
\ell(P)\le \ell(W)\le3.
\]
If $\ell(P)=\ell(W)$, the displayed walk in $\widetilde G$ already gives
\[
d_{\widetilde G}(\widetilde u,\widetilde v)
\le \ell(W)=d_G(u,v).
\]

Assume instead that $\ell(P)<\ell(W)$.  Then $\ell(P)\le2$, so the closed walk
$PW^{-1}$ has length at most five.  By
\cref{lem:short-closed-walk}, its lift beginning at $\widetilde u$ is closed.
Lift the shortest walk $P$ from $\widetilde u$.  If its endpoint is denoted
temporarily by
$\widetilde v'$, the remaining part of the closed lift is a lift of
$W^{-1}$ from $\widetilde v'$ back to $\widetilde u$.  Reversing that part
gives a lift of $W$ from $\widetilde u$ to $\widetilde v'$.  But the original
walk from $\widetilde u$ to $\widetilde v$ is also a lift of $W$ with the same
initial vertex.  Uniqueness of walk lifting therefore gives
$\widetilde v'=\widetilde v$.  Hence the lift of $P$ joins
$\widetilde u$ to $\widetilde v$, and
\[
d_{\widetilde G}(\widetilde u,\widetilde v)
\le \ell(P)=d_G(u,v).
\]
Together with the opposite inequality proved above, this gives the claim.
\end{proof}

By \cref{lem:local-distance}, the complete cost matrix entering the local
optimal-transport problem is unchanged by a $5$-cycle-preserving cover.  This
gives the curvature statement needed for the universal-cover argument.

\begin{proposition}[Local curvature preservation]\label{prop:cover-curvature}
Let $q:\widetilde G\to G$ be a $5$-cycle-preserving graph cover between
connected locally finite graphs, and let
$\{\widetilde x,\widetilde y\}\in E(\widetilde G)$, with
$x=q(\widetilde x)$ and $y=q(\widetilde y)$.  Then
\[
\kappa_\alpha^{\widetilde G}(\widetilde x,\widetilde y)
=
\kappa_\alpha^G(x,y)
\qquad(0\le\alpha\le1),
\]
where the superscript indicates the graph in which the curvature is computed.
Consequently,
\[
\kappa^{\widetilde G}(\widetilde x,\widetilde y)
=
\kappa^G(x,y).
\]
\end{proposition}

\begin{proof}
By the neighborhood-bijection condition in \cref{def:graph-cover}, a graph
cover preserves degrees and induces bijections
\[
N[\widetilde x]\longrightarrow N[x],
\qquad
N[\widetilde y]\longrightarrow N[y].
\]
These two bijections agree on their overlap.  Indeed, if
$\widetilde u\in N[\widetilde x]$ and
$\widetilde v\in N[\widetilde y]$ satisfy
$q(\widetilde u)=q(\widetilde v)$, then
\cref{lem:local-distance} gives
\[
d_{\widetilde G}(\widetilde u,\widetilde v)
=
d_G(q(\widetilde u),q(\widetilde v))
=0,
\]
so $\widetilde u=\widetilde v$.  Thus the two neighborhood bijections combine
to a bijection
\[
N[\widetilde x]\cup N[\widetilde y]
\longrightarrow
N[x]\cup N[y].
\]

Under this bijection, the lazy random-walk measures have the same masses
because degrees are preserved.  By \cref{lem:local-distance}, every distance
between a source in $N[\widetilde x]$ and a target in
$N[\widetilde y]$ equals the corresponding distance in $G$.  Hence the two
optimal-transport problems have the same marginals and the same transportation
costs.  Their Wasserstein distances, and therefore their
$\alpha$-Ollivier curvatures, are equal.  For $\alpha<1$, divide this equality
by $1-\alpha$ and let $\alpha\to1^-$ to obtain the Lin--Lu--Yau curvature
equality.
\end{proof}

\subsection{The universal cover and the fundamental-group bound}

We now apply local curvature preservation to the universal cover of
$\Xshort{5}(G)$.  The objective is to show that this universal cover has
finitely many vertices.  Positive Lin--Lu--Yau curvature gives a uniform
diameter bound by Bonnet--Myers, while the fact that the cover lies over a
finite graph gives a uniform degree bound.  These two facts force the cover to
be finite, and hence force $\pi_1(\Xshort{5}(G),o)$ to be finite for every
base vertex $o$.

\begin{proposition}[Short-cycle fundamental-group bound]\label{prop:short-cycle-pi1}
Let $G$ be a finite connected simple graph with at least one edge and
$\kappa_{\min}(G)>0$.  Put
\[
n:=|V(G)|,
\qquad
\Delta:=\max_{v\in V(G)}d_v,
\qquad
R:=\left\lfloor\frac{2}{\kappa_{\min}(G)}\right\rfloor.
\]
Then, for every base vertex $o\in V(G)$,
$\pi_1(\Xshort{5}(G),o)$ is finite and
\begin{equation}\label{eq:pi1-bound}
|\pi_1(\Xshort{5}(G),o)|
\le
\frac{1+\Delta\sum_{j=0}^{R-1}(\Delta-1)^j}{n}.
\end{equation}
Here $(\Delta-1)^0$ is interpreted as $1$, also when $\Delta=1$.
\end{proposition}

\begin{proof}
Since $G$ is finite and connected, $\Xshort{5}(G)$ is a connected finite CW
complex.  In particular, it is locally path connected and semilocally simply
connected, so the standard existence theorem for covering spaces gives a
universal covering; see \cite[Section~1.3]{Hatcher2002}.  Let
\[
p:\widetilde X\longrightarrow\Xshort{5}(G)
\]
be such a covering, and let
\[
\widetilde G:=\widetilde X^{(1)}
\]
be its one-skeleton.
The restriction of $p$ to $\widetilde G$ is a connected
graph cover of $G$, and it is $5$-cycle preserving, by
\cref{lem:CW-graph-cover}.  Moreover, \cref{def:graph-cover} gives
\[
\deg_{\widetilde G}(\widetilde v)
=\deg_G(p(\widetilde v))
\leq\Delta
\qquad(\widetilde v\in V(\widetilde G)),
\]
so $\widetilde G$ is locally finite.

For every edge $\{\widetilde a,\widetilde b\}\in E(\widetilde G)$,
\cref{prop:cover-curvature} gives
\[
\kappa^{\widetilde G}(\widetilde a,\widetilde b)
=
\kappa^G(p(\widetilde a),p(\widetilde b))
\ge
\kappa_{\min}(G)>0.
\]
The Lin--Lu--Yau Bonnet--Myers estimate therefore yields
\[
\diam(\widetilde G)
\le
\frac{2}{\kappa_{\min}(G)}.
\]
Since graph distances are integral, every vertex of $\widetilde G$ lies in
the ball of radius $R$ around any fixed vertex.

The same degree equality from \cref{def:graph-cover} shows that the maximum
degree of $\widetilde G$ is $\Delta$.  Fix $\widetilde o\in V(\widetilde G)$
and let $S_i$ be the set of vertices at distance $i$ from $\widetilde o$.
Then $|S_0|=1$, $|S_1|\leq\Delta$, and, for $i\geq1$, every vertex of $S_i$
has a neighbor in $S_{i-1}$ and therefore at most $\Delta-1$ neighbors in
$S_{i+1}$.  Counting the edges between consecutive distance layers gives
\[
|S_{i+1}|\leq(\Delta-1)|S_i|,
\qquad
|S_i|\leq\Delta(\Delta-1)^{i-1}\quad(i\geq1).
\]
Since $S_i=\varnothing$ for $i>R$, summing these bounds gives
\begin{equation}\label{eq:universal-cover-BFS}
|V(\widetilde G)|
\le
1+\Delta\sum_{j=0}^{R-1}(\Delta-1)^j.
\end{equation}
Thus $V(\widetilde G)$ is finite, so the covering has finitely many sheets.
Let $s=|p^{-1}(o)|$ be their number.  By the covering-space index theorem
\cite[Proposition~1.32]{Hatcher2002},
\[
s=
\bigl[\pi_1(\Xshort{5}(G),o):
p_*\pi_1(\widetilde X,\widetilde o)\bigr]
=|\pi_1(\Xshort{5}(G),o)|,
\]
because $\widetilde X$ is simply connected.  For every $v\in V(G)$, lifting a
fixed path from $o$ to $v$ gives a bijection
$p^{-1}(o)\to p^{-1}(v)$.  Thus every vertex fiber has $s$ elements, and the
$n$ vertex fibers partition $V(\widetilde G)$.  Consequently,
\[
|\pi_1(\Xshort{5}(G),o)|
=s=\frac{|V(\widetilde G)|}{n}.
\]
Dividing \eqref{eq:universal-cover-BFS} by $n$ proves
\eqref{eq:pi1-bound}.
\end{proof}

The bound has the following immediate formulation in GLMY terms.

\begin{corollary}[Five-cycle normal generation up to a finite quotient]
\label{cor:GLMY-pi1-pentagonal}
Under the hypotheses of \cref{prop:short-cycle-pi1}, for every base vertex
$o\in V(G)$,
\[
\pi_1^{\mathrm{GLMY}}(G,o)/N_5(G,o)
\]
is finite and satisfies the bound in \eqref{eq:pi1-bound}.
\end{corollary}

Indeed, this is exactly \eqref{eq:pentagonal-normal-quotient} together with
\cref{prop:short-cycle-pi1}.

\section{Examples and an extremal curvature problem}
\label{sec:examples-extremal}

This section presents examples that clarify the sharpness and scope of the
preceding results.  The five-cycle attains the curvature threshold in
\cref{thm:half-vanishing} and shows that cycles of length five are necessary
in \cref{prop:short-cycle-pi1}.  The graph $G_7$ shows that positive
Lin--Lu--Yau curvature is compatible with nonzero first GLMY path homology and an
infinite nonabelian GLMY fundamental group.  Cartesian products of copies of
$C_5$ provide positively curved graphs with nonzero higher-dimensional GLMY
path homology.  The Cartesian-product examples provide lower bounds for the
extremal curvature quantity introduced in the final subsection.

\subsection{The sharp five-cycle}\label{sec:examples}

The five-cycle simultaneously shows that the constant in
\cref{thm:half-vanishing} is sharp and that the filling length five in
\cref{prop:short-cycle-pi1} cannot be reduced to four.

Let $C_5$ have cyclically ordered vertices $0,1,2,3,4$, and take $0$ as the
base vertex.  It contains no simple $3$- or $4$-cycle, and hence
\[
\Xshort{4}(C_5)=C_5.
\]
By \eqref{eq:GLMY-pi1-X4},
\[
\pi_1^{\mathrm{GLMY}}(C_5,0)
\cong\pi_1(\Xshort{4}(C_5),0)
=\pi_1(C_5,0)
\cong\Z.
\]
The unique simple $5$-cycle represents a generator up to inversion and hence
normally generates the whole group.  Thus
\[
N_5(C_5,0)=\pi_1^{\mathrm{GLMY}}(C_5,0),
\qquad
\pi_1^{\mathrm{GLMY}}(C_5,0)/N_5(C_5,0)=1.
\]
Equivalently, $\Xshort{5}(C_5)$ is a disk, so
$\pi_1(\Xshort{5}(C_5),0)=1$.

Apply \eqref{eq:curvature-LP} to the edge $\{0,1\}$, with $g(0)=1$ and $g(1)=0$.
Since both degrees are two,
\[
\av_0g-\av_1g
=\frac{g(1)+g(4)-g(0)-g(2)}2
=\frac{g(4)-g(2)-1}{2}.
\]
The vertices $4$ and $2$ are at distance two, so
$g(4)-g(2)\leq2$.  Equality is attained, for example, by
\[
(g(0),g(1),g(2),g(3),g(4))=(1,0,0,1,2).
\]
Thus the maximum in \eqref{eq:curvature-LP} is $1/2$, and vertex
transitivity gives
\[
\kappa_{\mathrm{LLY}}(e)=\frac12
\qquad(e\in E(C_5)).
\]
Thus $\pi_1^{\mathrm{GLMY}}(C_5,0)\cong\Z$ although every edge of $C_5$ has positive
Lin--Lu--Yau curvature, whereas
$\pi_1(\Xshort{5}(C_5),0)=1$.  In particular, cycles of length five cannot be
omitted from the filling complex in \cref{prop:short-cycle-pi1}.

\subsection{A positively curved seven-vertex graph}
\label{sec:G7}

Let $G_7$ have vertex set $\{0,1,\ldots,6\}$ and edge set
\begin{equation}\label{eq:G7-edges}
\begin{split}
E(G_7)=\bigl\{&
\{0,1\},\{0,4\},\{1,2\},\{1,6\},\{2,3\},\\
&\{3,4\},\{3,5\},\{4,5\},\{5,6\}\bigr\}.
\end{split}
\end{equation}
A drawing is given in \cref{fig:G7}.

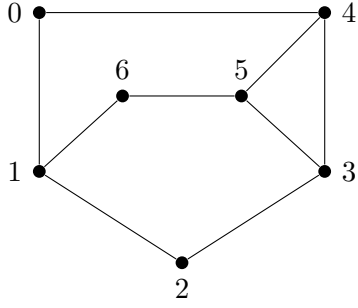
\begin{figure}[htbp]
\centering
\begin{tikzpicture}[
  scale=1.05,
  vertex/.style={circle,fill=black,inner sep=1.7pt}
]
\node[vertex,label=left:$0$]  (v0) at (-1.8, 1.15) {};
\node[vertex,label=left:$1$]  (v1) at (-1.8,-0.85) {};
\node[vertex,label=below:$2$] (v2) at ( 0.0,-2.00) {};
\node[vertex,label=right:$3$] (v3) at ( 1.8,-0.85) {};
\node[vertex,label=right:$4$] (v4) at ( 1.8, 1.15) {};
\node[vertex,label=above:$6$] (v6) at (-0.75,0.10) {};
\node[vertex,label=above:$5$] (v5) at ( 0.75,0.10) {};
\draw (v0)--(v1)--(v2)--(v3)--(v4)--(v0);
\draw (v3)--(v5)--(v4);
\draw (v1)--(v6)--(v5);
\end{tikzpicture}
\caption{The graph $G_7$.  Its unique simple $3$-cycle class vanishes in first
GLMY path homology, while two independent simple $5$-cycle classes remain.}
\label{fig:G7}
\end{figure}

We write
\[
F_2:=\langle a,b\mid\ \rangle
\]
for the free group of rank two; this should not be confused with the
two-element field $\mathbb F_2$.

\begin{proposition}\label{prop:G7}
The graph $G_7$ satisfies
\[
\kappa_{\min}^{\mathrm{LLY}}(G_7)=\frac16,
\qquad
\PathH_1(G_7;\R)\cong\R^2,
\qquad
\pi_1^{\mathrm{GLMY}}(G_7,0)\cong F_2.
\]
Moreover,
\[
N_5(G_7,0)=\pi_1^{\mathrm{GLMY}}(G_7,0),
\qquad
\pi_1^{\mathrm{GLMY}}(G_7,0)/N_5(G_7,0)=1.
\]
\end{proposition}

\begin{proof}
The graph has $9-7+1=3$-dimensional cycle space.  Its unique simple $3$-cycle is
\[
C_3=(3,4,5),
\]
and it has no simple $4$-cycle.  Consider the simple $5$-cycles
\[
C_5^{(1)}=(0,1,2,3,4),
\quad
C_5^{(2)}=(1,2,3,5,6),
\quad
C_5^{(3)}=(0,1,6,5,4).
\]
Their cycle vectors span the full three-dimensional cycle space.  Therefore
\[
\beta_1^{\mathrm{GLMY}}(G_7;\R)=3-1=2.
\]

Solving the finite linear programs in \eqref{eq:curvature-LP} gives
\[
\kappa(e)=
\begin{cases}
2/3,&e\in\bigl\{\{3,4\},\{3,5\},\{4,5\}\bigr\},\\
1/6,&e\in\bigl\{\{0,1\},\{0,4\},\{1,2\},\{1,6\},\{2,3\},\{5,6\}\bigr\}.
\end{cases}
\]
For exact certificates, orient the edges as listed in
\eqref{eq:G7-edges}.  The following maximizing functions are written as the
vectors $(g(0),\ldots,g(6))$.  Direct substitution verifies the edge
Lipschitz constraints, the normalization $g(x)=1$, $g(y)=0$, and the displayed
objective values.
\[
\begin{array}{c@{\qquad}c@{\qquad}c}
\toprule
(x,y)&(g(0),\ldots,g(6))&\av_xg-\av_yg\\
\midrule
(0,1)&(1,0,-1,0,1,0,-1)&5/6\\
(0,4)&(1,1,0,-1,0,-1,0)&5/6\\
(1,2)&(2,1,0,0,1,1,2)&5/6\\
(1,6)&(2,1,2,1,1,0,0)&5/6\\
(2,3)&(0,1,1,0,-1,-1,0)&5/6\\
(3,4)&(0,1,2,1,0,0,0)&1/3\\
(3,5)&(1,1,2,1,0,0,0)&1/3\\
(4,5)&(2,1,1,0,1,0,0)&1/3\\
(5,6)&(1,0,1,2,2,1,0)&5/6\\
\bottomrule
\end{array}
\]
Optimality can be checked without floating-point arithmetic.  The constraint
matrix consists of a signed incidence matrix together with the two endpoint
normalizations and is totally unimodular.  Hence every vertex of the bounded
normalized feasible polytope is integral.  The bounds
$|g(v)-1|\leq d(v,x)$ leave finitely many integral labelings; enumerating
those satisfying the edge inequalities gives maxima $5/6$ on the six
edges outside $C_3$ and $1/3$ on the three edges of $C_3$.  Thus
\eqref{eq:curvature-LP} gives the stated curvatures using exact rational
arithmetic.

To compute $\pi_1^{\mathrm{GLMY}}(G_7,0)$, take $0$ as the base vertex and
choose the spanning tree
\[
T=\bigl\{\{0,1\},\{1,2\},\{2,3\},\{3,4\},\{3,5\},\{5,6\}\bigr\}.
\]
The three non-tree edges are $\{0,4\}$, $\{1,6\}$, and $\{4,5\}$.  Orient
them respectively as $0\to4$, $1\to6$, and $4\to5$.  The corresponding based
loops are
\[
\begin{aligned}
a&=0\,4\,3\,2\,1\,0,\\
b&=0\,1\,6\,5\,3\,2\,1\,0,\\
c&=0\,1\,2\,3\,4\,5\,3\,2\,1\,0.
\end{aligned}
\]
The only $3$- or $4$-cycle relation is contributed by $C_3$.  After the tree
edges $\{3,4\}$ and $\{3,5\}$ are collapsed, its boundary word is
$c^{\pm1}$.  Hence the maximal-tree presentation is
\[
\pi_1^{\mathrm{GLMY}}(G_7,0)
\cong
\langle a,b,c\mid c\rangle
\cong
\langle a,b\mid\ \rangle
\cong F_2.
\]

The cycles $C_5^{(1)}$ and $C_5^{(2)}$ represent $a^{\pm1}$ and a conjugate
of $b^{\pm1}$, respectively.  They therefore normally generate $F_2$, which
proves
\[
N_5(G_7,0)=\pi_1^{\mathrm{GLMY}}(G_7,0)\cong F_2,
\qquad
\pi_1^{\mathrm{GLMY}}(G_7,0)/N_5(G_7,0)=1.
\]
Finally, abelianization gives
\[
\bigl(\pi_1^{\mathrm{GLMY}}(G_7,0)\bigr)_{\mathrm{ab}}
\cong F_2^{\mathrm{ab}}
\cong\Z^2,
\]
and \eqref{eq:GLMY-abelianization}, after tensoring with $\R$, identifies this
with $\PathH_1(G_7;\R)\cong\R^2$.
\end{proof}

Thus positive Lin--Lu--Yau curvature is compatible with nonzero first GLMY
path homology and an infinite nonabelian GLMY fundamental group.  At the same time,
the $5$-cycle normal subgroup is the whole group, in agreement with
\cref{cor:GLMY-pi1-pentagonal}.

\subsection{Positive curvature and higher-dimensional GLMY path homology}\label{sec:products}

Strict positivity of Lin--Lu--Yau curvature does not imply vanishing of
GLMY path homology in degrees $p\geq2$.  For each integer $r\geq1$, consider the
Cartesian product $T_r=C_5^{\square r}$ of $r$ copies of $C_5$.  As shown
below, every edge of $T_r$ has curvature $1/(2r)$.  Thus the minimum curvature
decreases as the number $r$ of factors increases.  This construction does not
address whether vanishing follows from a positive curvature lower bound that
depends on the homological degree.

For graphs $G$ and $H$, their Cartesian product $G\square H$ has vertex set $V(G)\times V(H)$, with
\[
(g,h)\sim(g',h')
\]
when either $g=g'$ and $\{h,h'\}\in E(H)$, or $h=h'$ and
$\{g,g'\}\in E(G)$.

The K\"unneth theorem for GLMY path homology
\cite[Theorem~4.7]{GrigoryanMuranovYau2017} states that, for finite digraphs
$D_1,D_2$ and a field $\F$,
\begin{equation}\label{eq:path-kunneth}
\PathH_p(D_1\square D_2;\F)
\cong
\bigoplus_{a+b=p}
\PathH_a(D_1;\F)\otimes_{\F}\PathH_b(D_2;\F).
\end{equation}
The isomorphism is induced by the cross product on the path chain complexes.
For undirected graphs, replacing every edge by the two opposite arrows
identifies the bidirected version of $G\square H$ with the digraph Cartesian
product of the bidirected versions of $G$ and $H$.  Thus
\eqref{eq:path-kunneth} applies to the products considered below.

\begin{theorem}[Cartesian products of five-cycles]\label{thm:pentagonal-tori}
Let
\[
T_r=C_5^{\square r},\qquad r\geq1.
\]
Then:
\begin{enumerate}[label=\textup{(\roman*)}]
\item $T_r$ has $5^r$ vertices, and every vertex has degree $2r$;
\item every edge of $T_r$ has Lin--Lu--Yau curvature $1/(2r)$;
\item for every field $\F$,
\[
\PathH_p(T_r;\F)
\cong
\begin{cases}
\F^{\binom rp},&0\leq p\leq r,\\
0,&p>r.
\end{cases}
\]
\end{enumerate}
\end{theorem}

\begin{proof}
Assertion~\textup{(i)} follows from the definition of the Cartesian product.

Lin, Lu, and Yau proved the Cartesian-product formula for regular graphs
\cite[Corollary~3.2]{LinLuYau2011}.  Suppose that $G$ and $H$ are regular of
degrees $d_G$ and $d_H$, respectively.  For an edge
$\{(g,h),(g',h)\}$ of $G\square H$ arising from
$\{g,g'\}\in E(G)$, the formula gives
\[
\kappa_{G\square H}\bigl((g,h),(g',h)\bigr)
=\frac{d_G}{d_G+d_H}\kappa_G(g,g').
\]
The corresponding formula for an edge on which only the $H$-coordinate
changes is obtained by interchanging $G$ and $H$.  Since every edge of $C_5$
has curvature $1/2$, iterating the formula gives
\[
\kappa(T_r)=\frac{1}{2r}.
\]

For a finite digraph $D$ whose GLMY path homology vanishes in sufficiently
high degrees, define the Poincar\'e polynomial of its GLMY path homology by
\[
P_D(t;\F):=\sum_{p\geq0}\dim_{\F}\PathH_p(D;\F)t^p.
\]
The computation of the GLMY path homology of bidirected cycles in
\cite{LiMuranovYau2026} gives
\[
P_{C_5}(t;\F)=1+t.
\]
By the K\"unneth theorem for GLMY path homology
\cite[Theorem~4.7]{GrigoryanMuranovYau2017}, Poincar\'e polynomials are
multiplicative under Cartesian products.  Therefore
\[
P_{T_r}(t;\F)=(1+t)^r
=\sum_{p=0}^{r}\binom rp t^p,
\]
which gives the claimed homology groups.
\end{proof}

\begin{corollary}[No positive-curvature vanishing in high degree]\label{cor:no-high-vanishing}
For every integer $N\geq1$ and every field $\F$, there exists a finite connected simple graph $G_N$ such that
\[
\kappa_{\min}(G_N)>0
\qquad\text{and}\qquad
\PathH_N(G_N;\F)\neq0.
\]
One may take $G_N=T_N=C_5^{\square N}$.
\end{corollary}

The preceding examples motivate the extremal problem formulated in the next
subsection.

\subsection{An extremal curvature problem}\label{sec:problems}

For $p\geq1$ and a field $\F$, define the Lin--Lu--Yau extremal quantity
\[
K_p^{\mathrm{LLY}}(\F)
=\sup\left\{\kappa_{\min}^{\mathrm{LLY}}(G):
\begin{array}{l}
G\text{ finite, connected, and simple,}\\
\PathH_p(G;\F)\neq0
\end{array}
\right\}.
\]
By \cref{thm:pentagonal-tori},
\[
K_p^{\mathrm{LLY}}(\F)\geq\frac1{2p}.
\]

\begin{problem}\label{prob:Kp}
Determine $K_p^{\mathrm{LLY}}(\F)$.  Is
$K_p^{\mathrm{LLY}}(\F)=1/(2p)$ for every $p$?  If so, classify the extremal
graphs and determine whether $C_5^{\square p}$ is rigid among extremizers.
\end{problem}

For real coefficients, \cref{thm:half-vanishing} gives
\[
\sup\{\kappa_{\min}^{\mathrm{LLY}}(G):
G\text{ finite and }\PathH_1(G;\R)\neq0\}=\frac12.
\]
The five-cycle attains the supremum.  The value of
$K_p^{\mathrm{LLY}}(\F)$ for
$p\geq2$ is not determined here.

\section*{Acknowledgments}

Shuliang Bai was supported by the National Natural Science Foundation of China
(NSFC, Grant No.~12301434).  Jingyan Li was supported by the Beijing Natural
Science Foundation (International Scientists Project, Grant No.~IS25081) and
the National Natural Science Foundation of China (NSFC, General Program, Grant
No.~82573048).

\bibliographystyle{amsplain}
\bibliography{bib}

\providecommand{\bysame}{\leavevmode\hbox to3em{\hrulefill}\thinspace}
\providecommand{\MR}{\relax\ifhmode\unskip\space\fi MR }
\providecommand{\MRhref}[2]{%
  \href{http://www.ams.org/mathscinet-getitem?mr=#1}{#2}
}
\providecommand{\href}[2]{#2}
\begin{thebibliography}{10}

\bibitem{GrigoryanJimenezMuranov2018}
Alexander Grigor'yan, Rolando Jimenez, and Yuri Muranov, \emph{Fundamental
  groupoids of digraphs and graphs}, Czechoslovak Math. J. \textbf{68} (2018),
  no.~1, 35--65.

\bibitem{GrigoryanLinMuranovYau2014}
Alexander Grigor'yan, Yong Lin, Yuri Muranov, and Shing-Tung Yau,
  \emph{Homotopy theory for digraphs}, Pure Appl. Math. Q. \textbf{10} (2014),
  no.~4, 619--674.

\bibitem{GrigoryanMuranovYau2017}
Alexander Grigor'yan, Yuri Muranov, and Shing-Tung Yau, \emph{Homologies of
  digraphs and {K}{\"u}nneth formulas}, Comm. Anal. Geom. \textbf{25} (2017),
  no.~5, 969--1018.

\bibitem{GrigoryanLinMuranovYau2020}
Alexander~A. Grigor'yan, Yong Lin, Yuri~V. Muranov, and Shing-Tung Yau,
  \emph{Path complexes and their homologies}, J. Math. Sci. \textbf{248}
  (2020), no.~5, 564--599.

\bibitem{Hatcher2002}
Allen Hatcher, \emph{Algebraic topology}, Cambridge University Press,
  Cambridge, 2002.

\bibitem{Hehl2025}
Moritz Hehl, \emph{Graphs with {Lin--Lu--Yau} curvature at least one and
  regular bone-idle graphs}, Calc. Var. Partial Differential Equations
  \textbf{64} (2025), Article 196.

\bibitem{Hehl2026}
\bysame, \emph{Regular graphs with positive {Ollivier--Ricci} curvature}, Math.
  Ann. \textbf{394} (2026), Article 24.

\bibitem{HehlMunch2025}
Moritz Hehl and Florentin M{\"u}nch, \emph{Betti number estimates for
  non-negatively curved graphs}, 2025, arXiv:2503.11222.

\bibitem{KemptonMunchYau2021}
Mark Kempton, Florentin M{\"u}nch, and Shing-Tung Yau, \emph{A homology
  vanishing theorem for graphs with positive curvature}, Comm. Anal. Geom.
  \textbf{29} (2021), no.~6, 1449--1473.

\bibitem{LiMuranovYau2026}
Jingyan Li, Yuri Muranov, and Shing-Tung Yau, \emph{Methods of algebraic
  topology in graph theory}, International Press, 2026.

\bibitem{LinYou2026}
Huiqiu Lin and Zhe You, \emph{Graphs with positive {Lin--Lu--Yau} curvature
  without quadrilaterals}, Calc. Var. Partial Differential Equations
  \textbf{65} (2026), Article 221.

\bibitem{LinLuYau2011}
Yong Lin, Linyuan Lu, and Shing-Tung Yau, \emph{{Ricci} curvature of graphs},
  Tohoku Math. J. (2) \textbf{63} (2011), no.~4, 605--627.

\bibitem{McShane1934}
E.~J. McShane, \emph{Extension of range of functions}, Bull. Amer. Math. Soc.
  \textbf{40} (1934), no.~12, 837--842.

\bibitem{MunchWojciechowski2019}
Florentin M{\"u}nch and Rados{\l}aw~K. Wojciechowski, \emph{{Ollivier Ricci}
  curvature for general graph laplacians: Heat equation, laplacian comparison,
  non-explosion and diameter bounds}, Adv. Math. \textbf{356} (2019), 106759.

\bibitem{Ollivier2009}
Yann Ollivier, \emph{{Ricci} curvature of {Markov} chains on metric spaces}, J.
  Funct. Anal. \textbf{256} (2009), no.~3, 810--864.

\bibitem{Villani2003}
C{\'e}dric Villani, \emph{Topics in optimal transportation}, Graduate Studies
  in Mathematics, vol.~58, American Mathematical Society, Providence, RI, 2003.

\end{thebibliography}

\end{document}